\documentclass[12pt]{article}

\usepackage{amsmath,amssymb,amsthm,mathtools,bbm}
\usepackage[hmargin=1in,vmargin=1in]{geometry}

\usepackage{enumitem}

\newtheorem{theorem}{Theorem}[section]
\newtheorem{lemma}[theorem]{Lemma}
\newtheorem{proposition}[theorem]{Proposition}

\theoremstyle{remark}

\DeclareMathOperator{\tr}{tr}

\usepackage[usenames,dvipsnames]{xcolor}
\usepackage[colorlinks=true,
  linkcolor=teal!60!black,
  citecolor=PineGreen,
  urlcolor=RedViolet]{hyperref}

\hypersetup{
  colorlinks   = true, 
  urlcolor     = blue, 
  linkcolor    = blue, 
  citecolor   = red, 
  bookmarksopen=true
}

\begin{document}

\title{On Learning-Curve Monotonicity for Maximum Likelihood Estimators}
\author{Mark Sellke \and Steven Yin}
\date{}
\maketitle

\begin{abstract}
The property of learning-curve monotonicity, highlighted in the recent papers \cite{VieringLoog2019,loog2019minimizers,viering2022shape}, describes algorithms which only improve in average performance given more data, for any underlying data distribution within a given family.
We establish the first nontrivial monotonicity guarantees for the maximum likelihood estimator in a variety of well-specified parametric settings.
For sequential prediction with log loss, we show monotonicity (in fact complete monotonicity) of the forward KL divergence for Gaussian vectors with unknown covariance and either known or unknown mean, as well as for Gamma variables with unknown scale parameter.
The Gaussian setting was explicitly highlighted as open in the aforementioned works, even in dimension $1$.
Finally we observe that for reverse KL divergence, a folklore majorization trick from \cite{marshall1965inequality} yields monotonicity for very general exponential families.

All results in this paper were derived by variants of GPT-5.2 Pro.
Humans did not provide any proof strategies or intermediate arguments, but only prompted the model to continue developing additional results, and verified and transcribed its proofs.

\end{abstract}

\setcounter{tocdepth}{1}
\tableofcontents

\section{Introduction}

In statistics and machine learning, the classical notion of a \emph{learning curve} describes how the performance of an algorithm evolves as it is exposed to more data. 
As highlighted in \cite{VieringLoog2019,loog2019minimizers,viering2022shape}, an appealing and intuitive property is \emph{monotonicity} in the data: the average performance of a learning algorithm should only improve as more data becomes available.
Although it is intuitively natural to expect such behavior, \cite{VieringLoog2019} constructed simple examples of non-monotone learning curves for mis-specified models, where the data generating distribution falls outside of the model class used for prediction (see also \cite{loog2023also}).

Does monotonicity hold for any commonly used estimators?
We focus in particular on the maximum likelihood estimator under well-specification, addressing a question of \cite{viering2022shape}:

\begin{center}
\parbox{0.85\textwidth}{
\emph{
...we wonder whether maximum likelihood estimators for well-specified
models behave monotonically. Likelihood estimation, being
a century-old, classical technique, has been heavily
studied, both theoretically and empirically. In much of the
theory developed, the assumption that one is dealing with
a correctly specified model is common, but we are not
aware of any results that demonstrate that better models are
obtained with more data.}
}
\end{center}

A particularly simple and fundamental instantiation of this general question was raised in \cite{VieringLoog2019}: fitting a single Gaussian.
For IID samples from a Gaussian distribution with known unit variance but unknown mean, the $n$-sample MLE (i.e. the sample mean) differs from the true mean by a centered Gaussian with variance $\frac{1}{n}$; thus monotonicity holds in any reasonable sense.
However monotonicity is much less clear for a Gaussian with known mean and \emph{unknown} variance.
Here we let $z_1,z_2,\dots$ be IID from a one-dimensional centered normal distribution $\mathcal N(0,v_\ast)$ with unknown variance $v_\ast>0$.
Then the maximum likelihood estimator for $v_*$ is $\hat v_n = \frac1n\sum_{i=1}^n z_i^2$.
We consider two risk measures for this approximation: the expected forward and reverse Kullback--Leibler (KL) divergence, given respectively by
\[
\mathcal E_n = \mathbb E\bigl[D_{\mathrm{KL}}(\mathcal N(0,v_\ast),\mathcal N(0,\hat v_n))\bigr],
\qquad
\widetilde{\mathcal E}_n = \mathbb E\bigl[D_{\mathrm{KL}}(\mathcal N(0,\hat v_n),\mathcal N(0,v_\ast))\bigr].
\]
In Example II therein, \cite{VieringLoog2019} demonstrated that a mis-specified analog of $\mathcal E_n$ (with $z_i$ IID from a non-Gaussian distribution) could be non-monotone, and went on to ask:
\begin{center}
\parbox{0.85\textwidth}{
\emph{
...this raises the issue
to what extent well-specified statistical models can actually be proven to behave monotonically. For
instance, is Example II monotone if the problem is well-specified?
}
}
\end{center}
\cite{loog2019minimizers} similarly remarked that monotonicity is unclear for Gaussians when both the mean and variance are unknown.
These works focused primarily on the forward KL risk $\mathcal E_n$, motivated by sequential prediction.

We establish monotonicity of the MLE in several well-specified settings. For forward KL divergence, we prove monotonicity for Gaussian vectors with unknown covariance and either known or unknown mean, as well as for Gamma distributions with unknown scale parameter but known shape (e.g.\ exponential random variables). 
In fact we show a much stronger property known as complete monotonicity, which implies that the marginal value of the $n$-th datapoint is not only positive but also strictly decreases with $n$.
For reverse KL divergence, we observe that monotonicity holds in the extremely general setting of exponential families, a fact related to known results from e.g.\ \cite{marshall1965inequality} and \cite[Chapter 2, Problem 34(b)]{cover1999elements}.

\subsection*{Statement on AI Assistance}

The proof of Theorem~\ref{thm:main}, which directly addresses the question of \cite{VieringLoog2019}, was originally due to an unreleased prototype of a longer-thinking-time version of GPT-5.2 Pro.
The model was asked to solve the $1$-dimensional problem raised in \cite{VieringLoog2019}, and came back to us with a correct proof in the more general setting of multivariate centered Gaussians.
With this proof in context, the now-public version of GPT-5.2 Pro was then able to:
\begin{enumerate}
    \item
    Transcribe the proof of Theorem~\ref{thm:main} into Latex form.
    \item 
    Locate references for standard facts on Gamma variables and functions (the original AI output proved identities such as Equation \eqref{eq:trigamma_series} and Lemma~\ref{lem:ElogGamma} from scratch).
    \item
    Extend the Gaussian results to unknown mean and reverse KL divergence, when we asked if such extensions were possible (see Theorem~\ref{thm:full_main} and Propositions~\ref{prop:reverse_known_mean} and \ref{prop:reverse_unknown_mean}).
    \item 
    Extend the reverse KL results to Poisson and Binomial distributions using a convexity argument, when we asked if this was possible (GPT-5.2 Pro also pointed out that the expected forward KL divergences are always infinite in both cases).
    \item 
    Extend this convexity proof to general exponential families, when asked whether it could be pushed further (see Proposition~\ref{prop:reverse_exp_family_monotone}).
    \item 
    Locate the closely related \cite[Chapter 2, Problem 34(b)]{cover1999elements}, when asked whether Proposition~\ref{prop:reverse_exp_family_monotone} was already known.
\end{enumerate}

Having obtained these results, we wanted to test how the public GPT-5.2 Pro model would fare without being given the proof of Theorem~\ref{thm:main} in the first step.
Impressively, in its first attempt at the same initial query (itself an AI-generated restatement of the open problem from \cite{VieringLoog2019}), the public model also successfully proved Theorem~\ref{thm:main}.
We then reproduced the other core results of this paper (Theorems~\ref{thm:full_main} and \ref{thm:gamma_forward_KL}, and Proposition~\ref{prop:reverse_exp_family_monotone}) in our first attempt using only the following three follow-up prompts (without any retries or prompt rewriting):
\begin{itemize}
    \item 
    \emph{how about in higher dimensions, when the entire covariance matrix is unknown?}
    \item 
    \emph{what about for exponential/gamma/poisson/binomial? does anything still work there?}
    \item 
    \emph{what about reverse KL? do things work there?}
\end{itemize}

On the other hand, our unreleased prototype also proved a much stronger \emph{complete monotonicity} property of the forward KL divergence via an explicit Laplace transform representation.
This extension was provided first (without our asking for it) in the originally requested setting of centered Gaussian variables in dimension $1$.
With this proof in context, GPT-5.2 Pro subsequently generalized it further to higher dimensional Gaussians, unknown means, and Gamma variables; see the discussion around Equation~\eqref{eq:laplace-formula} and Section~\ref{sec:laplace_monotonicity}.

In summary, the human contributions to this paper (aside from the development of GPT-5.2 Pro and its internal variant) were as follows:
\begin{enumerate}[label=(\alph*)]
    \item 
    Prompting the model to continue generalizing its results. 
    We did not work out any of these extensions ourselves before prompting the model, but just asked simple and natural followup questions.
    \item
    Checking the proofs written by GPT-5.2 Pro.
    To our knowledge, the only mathematical mistake that needed fixing occurred in \eqref{eq:trigamma_upper_simple}. GPT-5.2 Pro's initial proof flipped an inequality sign when bounding a sum by an integral, yielding an overly optimistic~bound.
    However this mistake did not affect any downstream arguments.
    \item 
    Editing and rearranging the writing for clarity and style, and writing the introduction. 
\end{enumerate}

Of course, all proofs, citations, and other claims were verified by the human authors, and we take full responsibility for their correctness.

\subsection{Main Results}

Let $Z_1,Z_2,\dots$ be IID from a distribution $P_\ast=P_{\theta_*}$ which belongs to a parametric family
$\{P_\theta:\theta\in\Theta\}$ with densities $\{p_\theta\}_{\theta\in\Theta}$ (with respect to a fixed base measure).
Given data $Z_{1:n}:=(Z_1,\dots,Z_n)$, the MLE is any maximizer
\[
\widehat\theta_n \in \arg\max_{\theta\in\Theta}\ \sum_{i=1}^n \log p_\theta(Z_i).
\]
We consider only settings where the MLE is almost surely unique, and study the expected accuracy of the fitted model $P_{\widehat\theta_n}$ as a function of $n$.

\medskip
\noindent\textbf{Learning-curve risks.}
To measure the estimation error between $\widehat\theta_n$ and $\theta_*$, we consider both directions of the Kullback--Leibler divergence
\[
D_{\mathrm{KL}}(P\|Q):=\mathbb E_{X\sim P}\!\left[\log\frac{dP}{dQ}(X)\right].
\]
Namely we define the \emph{forward} and \emph{reverse} KL learning curves of the MLE by
\[
\mathcal E_n \;:=\; \mathbb E\!\left[D_{\mathrm{KL}}(P_\ast\,\|\,P_{\widehat\theta_n})\right],
\qquad
\widetilde{\mathcal E}_n \;:=\; \mathbb E\!\left[D_{\mathrm{KL}}(P_{\widehat\theta_n}\,\|\,P_\ast)\right],
\]
where the expectation is over the training sample $Z_{1:n}$.
We say the learning curve is monotone for forward KL divergence if $\mathcal E_{n+1}\le \mathcal E_n$ for
all $n$ where the former is finite, and strictly so if the inequality is always strict; similarly for $\widetilde{\mathcal E}_n$.

We note that in any parametric family with densities $p_\theta$, the forward KL divergence is the excess expected log-loss of the prediction $P_{\widehat\theta_n}$ applied to IID data from $P_*$, relative to the true model, up to an
additive constant independent of the estimator:
\[
D_{\mathrm{KL}}(P_\ast\|P_{\widehat\theta_n})
=
\mathbb E_{Z\sim P_\ast}[-\log p_{\widehat\theta_n}(Z)] \;-\; \mathbb E_{Z\sim P_\ast}[-\log p_{\theta_\ast}(Z)].
\]
This sequential prediction perspective is the one taken in \cite{VieringLoog2019}.
We will however work only with $D_{\mathrm{KL}}$, which has the advantage of being invariant under reparametrization.

\medskip
\noindent\textbf{Gaussian models.}
Fix a dimension $d\ge 1$. In the centered multivariate Gaussian model, we observe $X_1\dots,X_n\stackrel{IID}{\sim}\mathcal N(0,\Sigma_\ast)$ with
unknown strictly positive definite covariance $\Sigma_\ast\in \mathbb S^d_{++}$.
Given these observations, the MLE for the covariance $\Sigma_\ast$ is given by
\[
\widehat\Sigma_n=\frac1n\sum_{i=1}^n X_iX_i^\top.
\]
The corresponding forward-KL learning curve is
\[
\mathcal E_{n,d}(\Sigma_\ast)
:=\mathbb E\!\left[D_{\mathrm{KL}}\!\big(\mathcal N(0,\Sigma_\ast)\,\|\,\mathcal N(0,\widehat\Sigma_n)\big)\right].
\]
In Section~\ref{sec:forward_gaussian} we show the following monotonicity property of $\mathcal E_{n,d}$.

\begin{theorem}
\label{thm:main}
Fix $d\ge 1$ and $\Sigma_* \in \mathbb S^d_{++}$.
If $n>d+1$, then the forward KL risk is independent of $\Sigma_*$ and given by
\begin{equation}
\label{eq:main_closed_form}
\mathcal E_{n,d}(\Sigma_\ast)
= \frac12\bigl(f_d(n)-d\bigr),
\end{equation}
where $\psi$ is the digamma function (see \eqref{eq:polygamma-functions}) and
\[
f_d(x)
:= \sum_{j=1}^d \psi\Bigl(\frac{x-j+1}{2}\Bigr)
- d\log(x/2)
+ \frac{xd}{x-d-1},
\qquad x>d+1.
\]
Moreover, $\mathcal E_{n,d}$ is data-monotone in the sense that
\[
\mathcal E_{n+1,d}(\Sigma_\ast) \;<\; \mathcal E_{n,d}(\Sigma_\ast) \;<\;\infty.
\]
If $n\le d+1$, then $\mathcal E_{n,d}(\Sigma_\ast)=+\infty$.
\end{theorem}

We also consider a full Gaussian model, in which we observe $X_i\sim\mathcal N(\mu_\ast,\Sigma_\ast)$ with
$(\mu_\ast,\Sigma_\ast)\in\mathbb R^d\times\mathbb S^d_{++}$ with both $\mu_\ast$ and $\Sigma_\ast$ unknown.
The MLE is then given by:
\[
\widehat\mu_n=\bar X:=\frac1n\sum_{i=1}^n X_i,
\qquad
\widehat\Sigma_n=\frac1n\sum_{i=1}^n (X_i-\bar X)(X_i-\bar X)^\top.
\]
We write $\mathcal E^{\mathrm{full}}_{n,d}(\mu_\ast,\Sigma_\ast)$ for the corresponding forward-KL learning curve of
$(\widehat\mu_n,\widehat\Sigma_n)$ (see \eqref{eq:full_risk_whitened}).
Here the finiteness threshold shifts to $n>d+2$.
The next result is also shown in Section~\ref{sec:forward_gaussian}.

\begin{theorem}
\label{thm:full_main}
Fix $d\ge 1$, $\mu_\ast\in\mathbb R^d$, and $\Sigma_* \in \mathbb S^d_{++}$.
If $n>d+2$, then the forward KL risk is independent of $(\mu_*,\Sigma_*)$ and given by
\begin{equation}
\label{eq:full_main_closed_form}
\mathcal E^{\mathrm{full}}_{n,d}(\mu_\ast,\Sigma_\ast)
=\frac12\bigl(g_d(n)-d\bigr),
\end{equation}
where $\psi$ is the digamma function and
\[
g_d(x)
:= \sum_{j=1}^d \psi\Bigl(\frac{x-j}{2}\Bigr)
- d\log(x/2)
+ \frac{(x+1)d}{x-d-2},
\qquad x>d+2.
\]
Moreover, $\mathcal E^{\mathrm{full}}_{n,d}$ is data-monotone in the sense that
\[
\mathcal E^{\mathrm{full}}_{n+1,d}(\mu_\ast,\Sigma_\ast)
< \mathcal E^{\mathrm{full}}_{n,d}(\mu_\ast,\Sigma_\ast)
< \infty.
\]
If $n\le d+2$, then $\mathcal E^{\mathrm{full}}_{n,d}(\mu_\ast,\Sigma_\ast)=+\infty$.
\end{theorem}

In fact, both Gaussian settings can be reduced to the isotropic case $P_*=\mathcal N(0,I_d)$ via coordinate changes, which explains why the above formulas for $\mathcal E_{n,d}$ and $\widetilde{\mathcal E}_{n,d}$ have no dependence on $\mu_*$ or $\Sigma_*$.
We derive the explicit polygamma function formulas in Theorems~\ref{thm:main} and~\ref{thm:full_main}, and deduce monotonicity using derivative estimates for the trigamma function $\psi_1=\psi'$.

\medskip
\noindent\textbf{Gamma-family scale estimation.}
A similar monotonicity extends to Gamma random variables.
We fix a known shape parameter $\alpha>0$ and seek to estimate the scale parameter $\theta_*$ from IID observations $X_i\sim\mathrm{Gamma}(\alpha,\theta_*)$. The MLE is
\[
\widehat\theta_n=\frac{1}{n\alpha}\sum_{i=1}^n X_i.
\]
In Section~\ref{sec:forward_gamma} we prove the following.

\begin{theorem}
\label{thm:gamma_forward_KL}
If $n\alpha>1$, then
\[
\mathbb E\Big[D_{\mathrm{KL}}(\mathrm{Gamma}(\alpha,\theta_\ast)\,\|\,\mathrm{Gamma}(\alpha,\widehat\theta_n))\Big]
=h_{\alpha}(n)=
\alpha\Big(\psi(n\alpha)-\log(n\alpha)+\frac{1}{n\alpha-1}\Big),
\]
where $\psi$ is the digamma function, and this quantity is strictly decreasing in $n$.
If $n\alpha\le 1$, the expectation is $+\infty$.
\end{theorem}

The special case $\alpha=1$ of Theorem~\ref{thm:gamma_forward_KL} recovers scale estimation for exponential random variables.
Also notably, the case $2\alpha=d\in\mathbb N$ is equivalent to estimating the scale of a centered Gaussian which is \emph{known to be isotropic}, i.e. of the form $\mathcal N(0,\theta_* I_d)$. Indeed here the squared norms $\|Z_i\|^2$ are IID $\mathrm{Gamma}(d/2,2\theta_*)$ and are sufficient statistics for $\theta_*$.

\medskip
\noindent\textbf{Complete Monotonicity.}
An unreleased prototype of a longer-thinking-time version of GPT-5.2 Pro completed the proofs of monotonicity in Theorems~\ref{thm:main}, \ref{thm:full_main}, and \ref{thm:gamma_forward_KL} using a different method.
This approach rewrites the explicit functions above as Laplace transforms of positive kernels, i.e. in the form
\begin{equation}
\label{eq:laplace-formula}
L(x)=\int_0^{\infty} e^{-xt} d\mu(t)
\end{equation}
where $\mu(t)$ is a positive finite measure on $[0,\infty)$ (for example a non-negative density).
We detail this in Section~\ref{sec:laplace_monotonicity}, and obtain the following. 

\begin{theorem}
\label{thm:complete-monotone}
The functions $f_d,g_d,h_{\alpha}$ each have representations of the form \eqref{eq:laplace-formula}, such that the integral converges absolutely on the respective half-lines $(d+1,\infty), (d+2,\infty),(1/\alpha,\infty)$.
In particular, each $L\in\{f_d,g_d,h_{\alpha}\}$ satisfies the complete monotonicity inequalities
\begin{equation}
\label{eq:alternating-derivatives}
(-1)^k L^{(k)}(x)>0
\end{equation}
for all $x$ in the domain of $L$ and all $k\in\mathbb N$, where $L^{(k)}$ denotes the $k$-th derivative.
\end{theorem}

The inequalities \eqref{eq:alternating-derivatives} follow directly from \eqref{eq:laplace-formula} by differentiation under the integral sign.
In fact the representation \eqref{eq:laplace-formula} famously characterizes all completely monotone functions; see \cite{bernstein1929fonctions,merkle2014completely} or \cite[Chapter 14]{lax2014functional}.
For integer $n$, the analogous discrete-difference properties follow by integration: for $k\in\mathbb N$ with all terms finite, one has for instance
\[
\binom{k}{0}\mathcal E_{n,d}
-
\binom{k}{1}\mathcal E_{n+1,d}
+
\binom{k}{2}\mathcal E_{n+2,d}
\dots
+ (-1)^k \binom{k}{k} \mathcal E_{n+k,d} > 0.
\]
In particular setting $k=1$ recovers ordinary monotonicity.
The case $k=2$ (i.e. that $\mathcal E_{n,d}-2\mathcal E_{n+1,d}+\mathcal E_{n+2,d}>0$) also has a natural interpretation: although the added value of the $n$-th datapoint is strictly positive, it is also strictly decreasing with $n$.

\medskip
\noindent\textbf{Reverse KL Divergence.}
We also consider the same questions under reverse KL divergence
\[
\widetilde{\mathcal E}_n=\mathbb E[D_{\mathrm{KL}}(P_{\widehat\theta_n}\|P_\ast)].
\]
In fact, monotonicity of the reverse KL divergence for the MLE is known to hold in the case of discrete random variables; here the model class consists of all probability distributions on the support, which means the MLE is the empirical distribution. 
See for instance \cite[Chapter 2, Problem 34(b)]{cover1999elements}.)
We observe (likely not for the first time) that the proof in the discrete case extends to the much more general setting of exponential families.
Recall that an exponential
family consists of densities of the form
\[
p_\theta(x)=\exp(\langle \theta, T(x)\rangle - A(\theta)), \quad\theta\in\Theta\subseteq\mathbb R^k
\]
with sufficient statistic $T$ and log-partition function $A$.
The family is regular if the mean map $\mu(\theta)=\mathbb E_\theta[T(X)]=\nabla A(\theta)$ is one-to-one and its
range $\mathcal M$ is open.
The convexity of the reverse KL divergence as a function of $\bar T_n:=\frac1n\sum_{i=1}^n T(X_i)$ yields the following. (See also \cite{marshall1965inequality,doob1971martingale} for older applications of the underlying convex domination idea.)

\begin{proposition}[Reverse-KL data monotonicity in exponential families]
\label{prop:reverse_exp_family_monotone}
Let $X_1,\dots,X_n\sim p_{\theta_\ast}$ IID in a regular exponential family.
Assume that for each $n$ in a range of interest the MLE exists, lies in the interior, and hence satisfies
$\widehat\mu_n:=\nabla A(\widehat\theta_n)=\bar T_n$ almost surely.
Then the expected reverse KL risk is nonincreasing:
\[
\mathbb E\!\left[D_{\mathrm{KL}}(p_{\widehat\theta_{n+1}}\,\|\,p_{\theta_\ast})\right]
\le
\mathbb E\!\left[D_{\mathrm{KL}}(p_{\widehat\theta_{n}}\,\|\,p_{\theta_\ast})\right].
\]
If in addition $A^\ast$ is strictly convex on $\mathcal M$ (equivalently, the family is minimal) and $T(X)$ is
non-degenerate under $p_{\theta_\ast}$, then the inequality is strict.
\end{proposition}

The Gaussian settings are both special cases of the above, with $T(x)=xx^{\top}$ and $T(x)=(x,xx^{\top})$ respectively.
Below we provide alternate proofs of monotonicity for the reverse KL risk of Gaussian estimation, which were derived by GPT-5.2 Pro prior to Proposition~\ref{prop:reverse_exp_family_monotone} and again come with explicit formulas involving the digamma function.
Of course, Proposition~\ref{prop:reverse_exp_family_monotone} encompasses other classical examples including Gamma, Poisson, and binomial random variables and many more.

\medskip
\begin{proposition}
\label{prop:reverse_known_mean}
Let $X_1,\dots,X_n\stackrel{IID}{\sim}\mathcal N(0,\Sigma_\ast)$ in $\mathbb R^d$, with MLE $\widehat\Sigma_n=\frac1n\sum_{i=1}^n X_iX_i^\top$.
If $n\ge d$, then the reverse KL risk is given by
\[
\mathbb E\Big[D_{\mathrm{KL}}\big(\mathcal N(0,\widehat\Sigma_n)\,\|\,\mathcal N(0,\Sigma_\ast)\big)\Big]
=\frac12\Big(d\log(n/2)-\sum_{j=1}^d \psi\!\Big(\frac{n-j+1}{2}\Big)\Big),
\]
and the expectation is strictly decreasing in $n$. If $n<d$, then the expectation is $+\infty$.
\end{proposition}

\medskip
\begin{proposition}
\label{prop:reverse_unknown_mean}
Let $X_1,\dots,X_n\sim\mathcal N(\mu_\ast,\Sigma_\ast)$ in $\mathbb R^d$, with MLEs
\[
\widehat\mu_n=\bar X,\qquad
\widehat\Sigma_n=\frac1n\sum_{i=1}^n (X_i-\bar X)(X_i-\bar X)^\top.
\]
If $n\ge d+1$, then the reverse KL risk is given by
\[
\mathbb E\Big[D_{\mathrm{KL}}\big(\mathcal N(\widehat\mu_n,\widehat\Sigma_n)\,\|\,\mathcal N(\mu_\ast,\Sigma_\ast)\big)\Big]
=\frac12\Big(d\log(n/2)-\sum_{j=1}^d \psi\!\Big(\frac{n-j}{2}\Big)\Big),
\]
and the expectation is strictly decreasing in $n$. If $n\le d$, then the expectation is $+\infty$.
\end{proposition}

\subsection{Other Related Work}

In the setting of sequential prediction, the realized risk $R_n$ at time $n$ is a function of the estimated parameter $\widehat\theta_n$ and the next sample $z_n$.
Given a well-specified prior for $\theta_*$, it is true in general that suitably defined Bayes-optimal algorithms have monotone learning curves, when averaged over the prior (see e.g. \cite[Section 4.6]{viering2022shape}).
For example, under log loss this is equivalent to the fact that posterior entropy decreases on average, a standard fact in information theory.
See \cite{haussler1991estimating} for other results in this direction.

Learning curves also play a central role in the modern literature on large-scale machine learning, often via empirical scaling laws that
relate test loss to dataset size and training compute \cite{kaplan2020scaling,hoffmann2022training}.
Relatedly, hyperparameter transfer methods aim to predict or accelerate progress along such curves across scales
\cite{yang2021tuning}.
The notable double descent phenomenon \cite{belkin2019reconciling,belkin2020two,nakkiran2021deep,hastie2022surprises,mei2022generalization,li2021minimum} is a surprising non-monotonicity property of the test loss relative to model size.
These works document striking regularities in overparameterized regimes, but they are largely
orthogonal to our focus of well-specified parametric models.
On the other hand, our Gaussian example does include high-dimensional linear regression with Gaussian covariates and errors, implying that the log-loss of the maximum likelihood linear model is data monotone, even when the noise level is one of the quantities to be estimated.

Classical theoretical analyses of learning curves in Bayesian/statistical-mechanics settings go back at least to
\cite{patarnello1987learning,seung1992statistical,opper1991calculation,amari1992four,amari1993universal}, with
additional viewpoints from PAC/VC-style estimation \cite{takahashi1993estimating,haussler1991estimating}.
The book \cite{devroye1997probabilistic} asked whether monotone Bayes-consistent algorithms exist; this was resolved positively in great generality by \cite{pestov2022universally}, and in a black-box way by \cite{BousquetEtAl2022}.
Our results show that the MLE suffices for monotonicity in several natural examples.

\section{Gamma functions and distributions}
\label{subsec:gamma_conventions}

The Gamma function and associated objects will play a key role in our proofs.
Recall that
\[
\Gamma(z)=\int_{0}^{\infty} t^{z-1}e^{-t}\,dt,\qquad z>0.
\]
The digamma and trigamma functions are defined by logarithmic differentiation:
\begin{equation}
\label{eq:polygamma-functions}
\psi(z)=\frac{d}{dz}\log\Gamma(z),
\qquad
\psi_1(z)=\psi'(z)=\frac{d^2}{dz^2}\log\Gamma(z).
\end{equation}
We use a standard series identity for $\psi_1$ (see e.g.\ \cite[Section 1.16, Equation (9)]{bateman1953higher}):
\begin{equation}
\label{eq:trigamma_series}
\psi_1(z) = \sum_{k=0}^\infty \frac{1}{(z+k)^2},\quad\forall z>0.
\end{equation}

This identity yields the following estimates which will be key in verifying monotonicity.

\begin{lemma}
\label{lem:trigamma_upper_simple}
Assuming $t>0$ in the first inequality and $t>1$ in the second, we have
\begin{equation}
\label{eq:trigamma_upper_simple}
\psi_1(t)<\frac{t+1}{t^2}<\frac{1}{t-1}.
\end{equation}
For $t>0$ we also have the lower bound:
\begin{equation}
\label{eq:trigamma_lower}
\psi_1(t)>\frac{1}{t}.
\end{equation}
\end{lemma}

\begin{proof}
The latter estimate in \eqref{eq:trigamma_upper_simple} is clear. For the former, we have from \eqref{eq:trigamma_series} that
\[
\psi_1(t)<\frac{1}{t^2}+\sum_{m=1}^\infty \frac{1}{(t+m)(t+m-1)}
= \frac{1}{t^2}+\frac{1}{t}
=
\frac{t+1}{t^2}
.
\]
For \eqref{eq:trigamma_lower}, we have 
\[
\psi_1(t) = \sum_{k=0}^\infty \frac{1}{(t+k)^2}
\;>\; \int_0^\infty \frac{ds}{(t+s)^2}
\;=\; \frac{1}{t}.\qedhere
\]
\end{proof}

Next we recall the Gamma distribution: for a shape parameter $\alpha>0$ and scale parameter $\theta>0$, we write $V\sim \mathrm{Gamma}(\alpha,\theta)$ if $V$ has density
\begin{equation}
\label{eq:gamma-density}
f(v)=\frac{1}{\Gamma(\alpha)\theta^{\alpha}} v^{\alpha-1}e^{-v/\theta}, \qquad v>0.
\end{equation}
An important special case is the chi-squared distribution $\chi^2_\nu = \mathrm{Gamma}(\nu/2,2)$ for $\nu>0$.
Recall that $\chi^2_\nu$ is the law of $Z_1^2+\dots+Z_{\nu}^2$ for IID standard Gaussian $Z_i$ when $\nu\in\mathbb N$.

\begin{lemma}[{\cite[Equation (17.29)]{JohnsonKotzBalakrishnan1994}}]
\label{lem:ElogGamma}
If $V\sim \mathrm{Gamma}(\alpha,\theta)$, then
\[
\mathbb E[\log V] = \psi(\alpha) + \log\theta.
\]
In particular, if $U\sim\chi^2_\nu=\mathrm{Gamma}(\nu/2,2)$, then
\[
\mathbb E[\log U] = \psi\Bigl(\frac{\nu}{2}\Bigr) + \log 2.
\]
\end{lemma}

\begin{lemma}[Reciprocal moment of a Gamma random variable]
\label{lem:gamma_inverse}
If $V\sim \mathrm{Gamma}(\alpha,\theta)$ and $\alpha>1$, then
\[
\mathbb E[V^{-1}] = \frac{1}{\theta(\alpha-1)}.
\]
If $\alpha\le 1$, then $\mathbb E[V^{-1}]=+\infty$.
In particular, if $U\sim\chi^2_\nu$ with $\nu>2$, then
\[
\mathbb E[U^{-1}] = \frac{1}{\nu-2}
\]
and $\mathbb E[U^{-1}]=\infty$ if $\nu\leq 2$.
\end{lemma}

\begin{proof}
Direct integration gives, for $\alpha>1$,
\begin{align*}
\mathbb E[V^{-1}]
=\int_0^\infty s^{-1}\frac{1}{\Gamma(\alpha)\theta^\alpha}s^{\alpha-1}e^{-s/\theta}\,ds
&=\frac{1}{\Gamma(\alpha)\theta^\alpha}\int_0^\infty s^{\alpha-2}e^{-s/\theta}\,ds
\\
&=\frac{\Gamma(\alpha-1)\theta^{\alpha-1}}{\Gamma(\alpha)\theta^\alpha}
=\frac{1}{\theta(\alpha-1)}.
\end{align*}
If $\alpha\le 1$, the integral diverges at $0$.
\end{proof}

\section{Gaussian and Gamma estimation under forward KL}
\label{sec:forward_gaussian}

\subsection{Coordinate Change Reductions for Gaussians}

Here we reduce both the known and unknown mean cases of Gaussian estimation to $P_*=\mathcal N(0,I_d)$ by coordinate change.

\begin{lemma}[Coordinate change for known mean]
\label{lem:whitening}
Let $Y_i := \Sigma_\ast^{-1/2}X_i$, so $Y_i\sim\mathcal N(0,I_d)$, and define
\[
\widehat S_n:=\frac1n\sum_{i=1}^n Y_iY_i^\top.
\]
Then $\widehat\Sigma_n=\frac{1}{n}\sum_{i=1}^n X_i X_i^{\top}=\Sigma_\ast^{1/2}\widehat S_n\Sigma_\ast^{1/2}$ and
\[
D_{\mathrm{KL}}\!\big(\mathcal N(0,\Sigma_\ast)\,\|\,\mathcal N(0,\widehat\Sigma_n)\big)
=
\frac12\Bigl(\log\det\widehat S_n+\tr(\widehat S_n^{-1})-d\Bigr)
\]
and so
\begin{equation}
\label{eq:En_d_in_terms_of_S}
\mathcal E_{n,d}(\Sigma_\ast)
=\frac12\Bigl(\mathbb E[\log\det\widehat S_n]+\mathbb E[\tr(\widehat S_n^{-1})]-d\Bigr).
\end{equation}
\end{lemma}

\begin{proof}
The covariance identity is immediate from $X_i=\Sigma_\ast^{1/2}Y_i$.
Moreover $\log\det\widehat\Sigma_n=\log\det\Sigma_\ast+\log\det\widehat S_n$.
For the trace term,
\[
\tr(\widehat\Sigma_n^{-1}\Sigma_\ast)
=\tr\big((\Sigma_\ast^{1/2}\widehat S_n\Sigma_\ast^{1/2})^{-1}\Sigma_\ast\big)
=\tr\big(\Sigma_\ast^{-1/2}\widehat S_n^{-1}\Sigma_\ast^{-1/2}\Sigma_\ast\big)
=\tr(\widehat S_n^{-1}),
\]
by cyclicity of trace. Substituting into $D_{\mathrm{KL}}(\mathcal N(0,\Sigma_\ast)\|\mathcal N(0,\widehat\Sigma_n))$ and taking expectations yields \eqref{eq:En_d_in_terms_of_S}.
\end{proof}

We omit the very similar proof for the unknown mean case, stated below.

\begin{lemma}[Coordinate change for the full model]
\label{lem:whitening_full}
Let $Y_i := \Sigma_\ast^{-1/2}(X_i-\mu_\ast)$ and $\bar Y:=\frac1n\sum_{i=1}^n Y_i$.
We now let
\[
\widehat S_n := \frac1n\sum_{i=1}^n (Y_i-\bar Y)(Y_i-\bar Y)^\top.
\]
Then $Y_i\sim\mathcal N(0,I_d)$ and the MLE is given by
\[
\widehat\mu_n-\mu_\ast=\Sigma_\ast^{1/2}\bar Y,
\qquad
\widehat\Sigma_n=\Sigma_\ast^{1/2}\widehat S_n\Sigma_\ast^{1/2}.
\]
Moreover, 
\[
D_{\mathrm{KL}}\!\big(\mathcal N(\mu_\ast,\Sigma_\ast)\,\|\,\mathcal N(\widehat\mu_n,\widehat\Sigma_n)\big)
=
\frac12\Bigl(\log\det\widehat S_n+\tr(\widehat S_n^{-1})+\bar Y^\top\widehat S_n^{-1}\bar Y-d\Bigr)
\]
and so
\begin{equation}
\label{eq:full_risk_whitened}
\mathcal E^{\mathrm{full}}_{n,d}(\mu_\ast,\Sigma_\ast)
=
\frac12
\Bigl(\mathbb E[\log\det\widehat S_n]+\mathbb E[\tr(\widehat S_n^{-1})]
+\mathbb E[\bar Y^\top\widehat S_n^{-1}\bar Y]
-d\Bigr).
\end{equation}
\end{lemma}

\subsection{Wishart Matrices}

Let $Y_1,\dots,Y_n\sim\mathcal N(0,I_d)$ IID and define the Wishart matrix
\[
W := \sum_{i=1}^n Y_iY_i^\top \sim \mathrm{Wishart}_d(n,I_d),
\qquad
\widehat S_n = \frac1n W.
\]
We rely on the following well-known consequence of Gram--Schmidt orthogonalization.

\begin{lemma}[{\cite[Lemma 7.2.1]{Anderson2003}}]
\label{lem:Bartlett}
Let $W\sim\mathrm{Wishart}_d(n,I_d)$ with $n\ge d$.
Then
\[
\det W \;\stackrel{d}{=}\; \prod_{j=1}^d \chi^2_{\,n-j+1},
\]
where the chi-squared variables on the right-hand side are independent.
\end{lemma}

\begin{proposition}
\label{prop:ElogdetW}
Let $W\sim\mathrm{Wishart}_d(n,I_d)$ with $n\ge d$.
Then
\[
\mathbb E[\log\det W]
= d\log 2 + \sum_{j=1}^d \psi\Bigl(\frac{n-j+1}{2}\Bigr).
\]
\end{proposition}

\begin{proof}
By Lemma~\ref{lem:Bartlett},
$\log\det W \stackrel{d}{=} \sum_{j=1}^d \log \chi^2_{n-j+1}$.
Apply Lemma~\ref{lem:ElogGamma} termwise.
\end{proof}

\begin{lemma}
\label{lem:schur}
Let $W\sim\mathrm{Wishart}_d(n,I_d)$ with $n\ge d$.
Then
\[
(W^{-1})_{11} \;\stackrel{d}{=}\; \frac{1}{\chi^2_{\,n-d+1}}.
\]
\end{lemma}

\begin{proof}
Let $Z$ be the $n\times d$ data matrix with rows $Y_1^\top,\dots,Y_n^\top$, so that $W=Z^\top Z$.
Write
\[
Z = [z_1\ Z_2],
\]
where $z_1\in\mathbb R^n$ is the first column and $Z_2\in\mathbb R^{n\times(d-1)}$ contains the remaining columns.
Then
\[
W = Z^\top Z
= \begin{pmatrix}
z_1^\top z_1 & z_1^\top Z_2\\
Z_2^\top z_1 & Z_2^\top Z_2
\end{pmatrix}
= \begin{pmatrix}
w_{11} & w_{12}^\top\\
w_{12} & W_{22}
\end{pmatrix}.
\]
The formula for the inverse of a block matrix is
\[
(W^{-1})_{11}
= \frac{1}{w_{11} - w_{12}^\top W_{22}^{-1}w_{12}}.
\]
It is well-known (see e.g.\ \cite[Theorem 3.2.10]{Muirhead1982}) that the denominator is a chi-squared random variable with the claimed number of degrees of freedom, as desired.
Explicitly, we may write
\[
w_{11}-w_{12}^\top W_{22}^{-1}w_{12}
= z_1^\top\left(I_n-Z_2(Z_2^\top Z_2)^{-1}Z_2^\top\right)z_1.
\]
The matrix 
\[
Q=I_n-Z_2(Z_2^\top Z_2)^{-1}Z_2^\top
\]
is the orthogonal projector onto $\mathrm{col}(Z_2)^\perp$ and has rank $n-d+1$ almost surely. Since $z_1\sim \mathcal N(0,I_n)$ and is independent of $Z_2$, it follows that $z_1^\top Q z_1 \sim \chi^2_{n-d+1}$ is equal in distribution to the sum of squares of $n-d+1$ IID standard Gaussians, as claimed.
\end{proof}

\begin{proposition}
\label{prop:EWinv}
Let $W\sim\mathrm{Wishart}_d(n,I_d)$ with $n\ge d$.
If $n\leq d+1$ then $\mathbb E[(W^{-1})_{11}]=\mathbb E[\tr(W^{-1})]=\infty$.
If $n>d+1$ then
\[
\mathbb E[W^{-1}]
=
\frac{I_d}{n-d-1},\qquad 
\mathbb E[\tr(W^{-1})]=\frac{d}{n-d-1}.
\]
\end{proposition}

\begin{proof}
Lemma~\ref{lem:schur} and Lemma~\ref{lem:gamma_inverse} give
$\mathbb E[(W^{-1})_{11}]=1/(n-d-1)$ for $n>d+1$ and $+\infty$ otherwise.
Positive definiteness implies the off-diagonal entries have finite expectation when the diagonal entries do.
Orthogonal invariance of $W$ implies $\mathbb E[W^{-1}]\propto I_d$ when the expectation exists.
Combining completes the proof.
\end{proof}

\subsection{Known mean: closed form and monotonicity}

\begin{proof}[Proof of Theorem~\ref{thm:main}]
We first show \eqref{eq:main_closed_form}.
This follows by combining Lemma~\ref{lem:whitening} with Proposition~\ref{prop:ElogdetW} and Proposition~\ref{prop:EWinv},
using $\widehat S_n=W/n$, which yields 
\[
\log\det\widehat S_n=\log\det W-d\log n,\quad 
\tr(\widehat S_n^{-1})=n\tr(W^{-1}).
\]
To complete the proof, it suffices to show $f_d'(x)<0$ for all $x>d+1$.
Differentiating yields:
\[
f_d'(x)
=\frac12\sum_{j=1}^d \psi_1\Bigl(\frac{x-j+1}{2}\Bigr)
-\frac{d}{x}
-\frac{d(d+1)}{(x-d-1)^2}.
\]
By \eqref{eq:trigamma_upper_simple}, we have
$\frac12\psi_1((x-j+1)/2) < \frac{1}{x-j-1}$, hence
\begin{align*}
f_d'(x)
< \sum_{j=1}^d \frac{1}{x-j-1} -\frac{d}{x}-\frac{d(d+1)}{(x-d-1)^2}
&\le d\Big(\frac{1}{x-d-1}-\frac{1}{x}-\frac{d+1}{(x-d-1)^2}\Big)
\\
&= \frac{-d(d+1)^2}{x(x-d-1)^2}<0.
\qedhere
\end{align*}
\end{proof}

\subsection{Unknown mean: closed form and monotonicity}

For a Gaussian with unknown mean, recentering around the empirical sample mean is essentially equivalent to removing a single datapoint. We detail this below.

\begin{lemma}
\label{lem:centered_scatter_wishart}
Let $Y_1,\dots,Y_n\sim\mathcal N(0,I_d)$ be IID and $\bar Y=\frac1n\sum_i Y_i$.
Define
\[
W_c:=\sum_{i=1}^n (Y_i-\bar Y)(Y_i-\bar Y)^\top.
\]
Then $W_c\sim\mathrm{Wishart}_d(n-1,I_d)$ and $W_c$ is independent of $\bar Y$.
\end{lemma}

\begin{proof}
Let $Z\in\mathbb R^{n\times d}$ be the data matrix with rows $Y_1^\top,\dots,Y_n^\top$.
Let $\mathbf 1_n\in\mathbb R^n$ be the all-ones vector and set $q_1:=\mathbf 1_n/\sqrt n$. Extend $q_1$
to an orthonormal basis $q_1,\dots,q_n$ and let $Q\in\mathbb R^{n\times n}$ be the orthogonal matrix with
rows $q_1^\top,\dots,q_n^\top$. Since $Z$ is Gaussian with IID $\mathcal N(0,1)$ entries, $QZ$ has
the same distribution and has independent rows. Writing
\[
QZ = \begin{pmatrix}\sqrt n\,\bar Y^\top\\ \widetilde Z\end{pmatrix},
\]
the matrix $\widetilde Z\in\mathbb R^{(n-1)\times d}$ has IID $\mathcal N(0,1)$ entries and is independent
of $\bar Y$. Moreover,
\[
W_c
= Z^\top\Bigl(I_n-\frac1n\mathbf 1_n\mathbf 1_n^\top\Bigr)Z
= Z^\top Q^\top \begin{pmatrix}0&0\\0&I_{n-1}\end{pmatrix} Q Z
= \widetilde Z^\top\widetilde Z,
\]
so $W_c\sim\mathrm{Wishart}_d(n-1,I_d)$ and is independent of $\bar Y$.
\end{proof}

\begin{proof}[Proof of Theorem~\ref{thm:full_main}]
We start with \eqref{eq:full_main_closed_form}.
By Lemma~\ref{lem:whitening_full} we may assume $Y_i\sim\mathcal N(0,I_d)$.
Let $W_c$ be as in Lemma~\ref{lem:centered_scatter_wishart}; then $\widehat S_n=W_c/n$.
The log-determinant term is given by Proposition~\ref{prop:ElogdetW} with $n-1$ in place of $n$.
The inverse-trace term is given by Proposition~\ref{prop:EWinv} with the same substitution.
For the mean term, independence gives
\[
\mathbb E[\bar Y^\top \widehat S_n^{-1}\bar Y]
=\tr\!\Big(\mathbb E[\widehat S_n^{-1}]\,\mathbb E[\bar Y\bar Y^\top]\Big)
=\tr\!\Big(\frac{n}{n-d-2}I_d\cdot \frac1n I_d\Big)
=\frac{d}{n-d-2}.
\]
As in the centered case, it remains to show $g_d'(x)<0$ for all $x>d+2$. Differentiating:
\[
g_d'(x)=\frac12\sum_{j=1}^d \psi_1\Bigl(\frac{x-j}{2}\Bigr)
-\frac{d}{x}-\frac{d(d+3)}{(x-d-2)^2}.
\]
By \eqref{eq:trigamma_upper_simple}, $\frac12\psi_1((x-j)/2)<\frac{1}{x-j-2}$, hence
\begin{align*}
g_d'(x)
< \Big(\sum_{j=1}^d \frac{1}{x-j-2}\Big)-\frac{d}{x}-\frac{d(d+3)}{(x-d-2)^2}
&\le d\Big(\frac{1}{x-d-2}-\frac{1}{x}-\frac{d+3}{(x-d-2)^2}\Big)
\\
&=\frac{-d\bigl(x+(d+2)^2\bigr)}{x(x-d-2)^2}
<0.
\qedhere
\end{align*}
\end{proof}

\subsection{Gamma Distributions under Forward KL}
\label{sec:forward_gamma}

Fix $\alpha>0$ and let $X_i\sim \mathrm{Gamma}(\alpha,\theta_\ast)$ with unknown scale $\theta_\ast>0$ and known shape $\alpha$.
The MLE is
\begin{equation}
\label{eq:gamma-MLE}
\widehat\theta_n=\frac{1}{n\alpha}\sum_{i=1}^n X_i.
\end{equation}
Indeed, the log-likelihood as a function of $\theta$ is
\[
\ell(\theta)=-n\alpha\log\theta-\theta^{-1}\sum_i X_i+C
\]
where $C$ does not depend on $\theta$; setting $\ell'(\theta)=0$ yields
\eqref{eq:gamma-MLE}.

\begin{proof}[Proof of Theorem~\ref{thm:gamma_forward_KL}]
For fixed shape $\alpha$, one computes
\[
D_{\mathrm{KL}}(\mathrm{Gamma}(\alpha,\theta_\ast)\,\|\,\mathrm{Gamma}(\alpha,\theta))
=
\alpha\Big(\log\frac{\theta}{\theta_\ast}+\frac{\theta_\ast}{\theta}-1\Big).
\]
Indeed, the log-density ratio for $X\sim\mathrm{Gamma}(\alpha,\theta_\ast)$ is
\[
\log\frac{f_{\theta_\ast}(X)}{f_{\theta}(X)}
=
\Big(-\alpha\log\theta_\ast-\frac{X}{\theta_\ast}\Big)-\Big(-\alpha\log\theta-\frac{X}{\theta}\Big)
=
\alpha\log\frac{\theta}{\theta_\ast}+X\Big(\frac{1}{\theta}-\frac{1}{\theta_\ast}\Big).
\]
Taking expectation and using $\mathbb E_{\theta_\ast}[X]=\alpha\theta_\ast$ gives
\[
D_{\mathrm{KL}}(\mathrm{Gamma}(\alpha,\theta_\ast)\,\|\,\mathrm{Gamma}(\alpha,\theta))
=\alpha\log\frac{\theta}{\theta_\ast}+\alpha\theta_\ast\Big(\frac{1}{\theta}-\frac{1}{\theta_\ast}\Big)
=\alpha\Big(\log\frac{\theta}{\theta_\ast}+\frac{\theta_\ast}{\theta}-1\Big).
\]
Let $S=\sum_i X_i\sim \mathrm{Gamma}(n\alpha,\theta_\ast)$, so 
\[
\widehat\theta_n=\frac{S}{n\alpha} \sim 
\mathrm{Gamma}\Big(n\alpha,\frac{\theta_\ast}{n\alpha}\Big).
\]
Then Lemma~\ref{lem:ElogGamma} gives $\mathbb E[\log(\widehat\theta_n/\theta_\ast)]=\psi(n\alpha)-\log(n\alpha)$.
Lemma~\ref{lem:gamma_inverse} gives (for $n\alpha>1$)
\[
\mathbb E\Big[\frac{\theta_\ast}{\widehat\theta_n}\Big]
=\theta_\ast\,\mathbb E[\widehat\theta_n^{-1}]
= \frac{\theta_\ast n\alpha}{\theta_\ast(n\alpha-1)}=\frac{n\alpha}{n\alpha-1}.
\]
Monotonicity follows by differentiating in $t=n\alpha$ and using \eqref{eq:trigamma_upper_simple}.
Concretely, define
\[
h(t):=\alpha\Big(\psi(t)-\log t+\frac{1}{t-1}\Big),\quad\forall t>1.
\]
Then $h(n\alpha)$ is the displayed risk, and
\[
h'(t)=\alpha\Big(\psi_1(t)-\frac{1}{t}-\frac{1}{(t-1)^2}\Big).
\]
By \eqref{eq:trigamma_upper_simple}, $\psi_1(t)<\frac{t+1}{t^2}
<\frac{1}{t}+\frac{1}{(t-1)^2}$ for $t>1$; hence $h'(t)<0$, implying the result.
\end{proof}

\section{Complete monotonicity of Forward KL Divergence}
\label{sec:laplace_monotonicity}

Here we finish the proofs of Theorems~\ref{thm:main}, \ref{thm:full_main}, and \ref{thm:gamma_forward_KL} in a different way.
Namely we represent the relevant functions as Laplace transforms of positive measures, i.e. in the form
\[
L(x)= \int_0^{\infty} e^{-xt} d\mu(t).
\]
This representation implies the much stronger property of \emph{complete monotonicity}: such a function $L$ satisfies $(-1)^k L^{(k)}(x)\geq 0$ where $L^{(k)}$ is the $k$-th derivative.

We will use the standard Laplace transform identities:
\begin{align}
\label{eq:laplace-1}
\frac{1}{x}&=\int_0^\infty e^{-xt}\,dt;
\\
\label{eq:laplace-2}
\frac{1}{(x-w)^2}=\int_0^\infty t\,e^{-(x-w)t}\,dt
&=\int_0^\infty t\,e^{-xt}e^{wt}\,dt, \quad\forall x>w.
\end{align}
By expanding the series \eqref{eq:trigamma_series} for $\psi_1$ we also obtain the Laplace transform representation of the trigamma function.

\begin{lemma}
\label{lem:trigamma_laplace}
For every $u>0$,
\begin{equation}
\label{eq:laplace-3}
\psi_1(u)=\int_0^\infty \frac{t\,e^{-ut}}{1-e^{-t}}\,dt.
\end{equation}
\end{lemma}

\begin{proof}
Starting from \eqref{eq:trigamma_series} and the identity
$\frac{1}{a^2}=\int_0^\infty t e^{-at}\,dt$ for $a>0$, we obtain
\[
\psi_1(u)=\sum_{k=0}^\infty \int_0^\infty t e^{-(u+k)t}\,dt
=\int_0^\infty t e^{-ut}\sum_{k=0}^\infty e^{-kt}\,dt
=\int_0^\infty \frac{t\,e^{-ut}}{1-e^{-t}}\,dt,
\]
where Tonelli's theorem justifies exchanging sum and integral.
\end{proof}

We now prove Theorem~\ref{thm:complete-monotone}, proceeding separately for each of the three cases.
In fact we will show the functions $-f_d',-g_d',-h_{\alpha}'$ are Laplace transforms of positive kernels as in \eqref{eq:laplace-formula}.
This implies they are completely monotone, hence so are $f_d,g_d,h_{\alpha}$.
(By definition of complete monotonicity it remains to verify their non-negativity; this follows from non-negativity of KL divergence.)
This yields Theorem~\ref{thm:complete-monotone} by again applying the equivalence between \eqref{eq:laplace-formula} and completely monotone functions in the opposite direction (see e.g. \cite[Chapter 14]{lax2014functional}).

\begin{proof}[Proof of Theorem~\ref{thm:complete-monotone} for $f_d$]
Fix $x>d+1$. Using \eqref{eq:laplace-1}, \eqref{eq:laplace-2}, \eqref{eq:laplace-3} with $w=d+1$
and the change of variables $s\mapsto 2t$, we compute
\begin{align*}
\frac12\psi_1\Big(\frac{x-j+1}{2}\Big)
&=\frac12\int_0^\infty \frac{s\,e^{-(x-j+1)s/2}}{1-e^{-s}}\,ds
=\int_0^\infty \frac{2t\,e^{-(x-j+1)t}}{1-e^{-2t}}\,dt.
\end{align*}
Summing over $j=1,\dots,d$ and using $\sum_{j=1}^d e^{(j-1)t}=\frac{e^{dt}-1}{e^{t}-1}$ gives
\[
\frac12\sum_{j=1}^d \psi_1\Big(\frac{x-j+1}{2}\Big)
=\int_0^\infty e^{-xt}\,
\frac{2t}{1-e^{-2t}}\cdot\frac{e^{dt}-1}{e^t-1}\,dt.
\]
Substituting these representations into $-f_d'(x)=\frac{d}{x}+\frac{d(d+1)}{(x-d-1)^2}-\frac12\sum_j\psi_1(\tfrac{x-j+1}{2})$ yields
\begin{equation}
\label{eq:laplace_kernel_fd}
-f_d'(x)=\int_0^\infty e^{-xt}\,B_d(t)\,dt,
\end{equation}
where the kernel is
\begin{equation}
\label{eq:Bd_def}
B_d(t)
:= d+d(d+1)t\,e^{(d+1)t}
-\frac{2t}{1-e^{-2t}}\cdot\frac{e^{dt}-1}{e^t-1},
\qquad t>0.
\end{equation}
The integral in \eqref{eq:laplace_kernel_fd} is absolutely convergent for $x>d+1$:
as $t\downarrow 0$ one has $B_d(t)=O(t)$ (hence no singularity at $0$), while as $t\to\infty$,
$B_d(t)=O(t e^{(d+1)t})$ and $e^{-xt}$ provides exponential decay since $x>d+1$.

It remains to show $B_d(t)>0$ for all $t>0$. First note
\[
\frac{2t}{1-e^{-2t}}=\frac{t e^t}{\sinh t}\le e^t
\qquad (t>0),
\]
since $\sinh t\ge t$ for $t\ge 0$. Hence
\[
B_d(t)\;\ge\; d+d(d+1)t\,e^{(d+1)t}
-e^t\cdot\frac{e^{dt}-1}{e^t-1}.
\]
Using $e^t\cdot \frac{e^{dt}-1}{e^t-1}=\sum_{k=1}^{d}e^{kt}$, we obtain the simpler lower bound
\[
B_d(t)\;\ge\;F_d(t)
:= d+d(d+1)t\,e^{(d+1)t}-\sum_{k=1}^d e^{kt}.
\]
Now $F_d(0)=0$, and differentiating gives
\[
F_d'(t)=d(d+1)e^{(d+1)t}\bigl(1+(d+1)t\bigr)-\sum_{k=1}^d k e^{kt}.
\]
Since $e^{kt}\le e^{dt}$ for $1\le k\le d$, we have
\[
\sum_{k=1}^d k e^{kt}\le e^{dt}\sum_{k=1}^d k=\frac{d(d+1)}{2}\,e^{dt}.
\]
Therefore,
\[
F_d'(t)\ge d(d+1)e^{dt}\Big(e^t(1+(d+1)t)-\frac12\Big).
\]
The right-hand side is strictly positive for $t>0$
because $e^t(1+(d+1)t)>1$.
Thus $F_d(t)>0$ for all $t>0$, and hence $B_d(t)\ge F_d(t)>0$ for all $t>0$.
\end{proof}


\begin{proof}[Proof of Theorem~\ref{thm:complete-monotone} for $g_d$]
Recall that for $x>d+2$ we have
\[
g_d'(x)
=\frac12\sum_{j=1}^d \psi_1\Big(\frac{x-j}{2}\Big)
-\frac{d}{x}
-\frac{d(d+3)}{(x-d-2)^2}.
\]
Fix $x>d+2$. Using \eqref{eq:laplace-1}, \eqref{eq:laplace-2}, \eqref{eq:laplace-3} with $w=d+2$ and the change of variables $s\mapsto 2t$, we compute
\[
\frac12\psi_1\Big(\frac{x-j}{2}\Big)
=\frac12\int_0^\infty \frac{s\,e^{-(x-j)s/2}}{1-e^{-s}}\,ds
=\int_0^\infty \frac{2t\,e^{-(x-j)t}}{1-e^{-2t}}\,dt.
\]
Summing over $j=1,\dots,d$ yields
\[
\frac12\sum_{j=1}^d \psi_1\Big(\frac{x-j}{2}\Big)
=\int_0^\infty e^{-xt}\,
\frac{2t}{1-e^{-2t}}\cdot\sum_{j=1}^d e^{jt}\,dt
=\int_0^\infty e^{-xt}\,
\frac{2t}{1-e^{-2t}}\cdot\frac{e^{(d+1)t}-e^{t}}{e^t-1}\,dt.
\]
Substituting into
\[
-g_d'(x)=\frac{d}{x}+\frac{d(d+3)}{(x-d-2)^2}-\frac12\sum_{j=1}^d \psi_1\Big(\frac{x-j}{2}\Big)
\]
gives the Laplace transform representation
\begin{equation}
\label{eq:laplace_kernel_gd}
-g_d'(x)=\int_0^\infty e^{-xt}\,\widetilde B_d(t)\,dt,
\end{equation}
where the kernel is
\[
\widetilde B_d(t)
:= d+d(d+3)t\,e^{(d+2)t}
-\frac{2t}{1-e^{-2t}}\cdot\sum_{j=1}^d e^{jt}
\;=\;
d+d(d+3)t\,e^{(d+2)t}
-\frac{2t}{1-e^{-2t}}\cdot\frac{e^{(d+1)t}-e^{t}}{e^t-1}.
\]
The integrand in \eqref{eq:laplace_kernel_gd} is again absolutely convergent for $x>d+2$.

We now show $\widetilde B_d(t)>0$ for all $t>0$. Using again that $\frac{2t}{1-e^{-2t}}\leq e^t$ for $t>0$ we find
\[
\widetilde B_d(t)\geq \widetilde F_d(t) := 
d+d(d+3)t\,e^{(d+2)t}-e^t\sum_{j=1}^d e^{jt}
=d+d(d+3)t\,e^{(d+2)t}-\sum_{k=2}^{d+1} e^{kt}.
\]
Then $\widetilde F_d(0)=0$, and differentiating gives
\[
\widetilde F_d'(t)
=d(d+3)e^{(d+2)t}\bigl(1+(d+2)t\bigr)-\sum_{k=2}^{d+1} k e^{kt}.
\]
Since $e^{kt}\le e^{(d+1)t}$ for $2\le k\le d+1$ and
\[
\sum_{k=2}^{d+1} k = \frac{(d+1)(d+2)}{2}-1=\frac{d(d+3)}{2},
\]
we have
\[
\sum_{k=2}^{d+1} k e^{kt}\le \frac{d(d+3)}{2}\,e^{(d+1)t}.
\]
Therefore,
\[
\widetilde F_d'(t)
\ge d(d+3)e^{(d+1)t}\Big(e^t(1+(d+2)t)-\frac12\Big)>0
\qquad (t>0),
\]
because $e^t(1+(d+2)t)\ge 1$ for all $t\ge 0$ and is strictly $>1$ for $t>0$.
Thus $\widetilde F_d(t)>0$ for all $t>0$, and hence $\widetilde B_d(t)\ge \widetilde F_d(t)>0$ for all $t>0$.
\end{proof}

\begin{proof}[Proof of Theorem~\ref{thm:complete-monotone} for $h_{\alpha}$]
For $t:=n\alpha$, the expected forward KL risk equals
\[
\mathcal R_{n,\alpha}(\theta_\ast)
:=\mathbb E\Big[D_{\mathrm{KL}}(\mathrm{Gamma}(\alpha,\theta_\ast)\,\|\,\mathrm{Gamma}(\alpha,\widehat\theta_n))\Big]
= \alpha\,h(t);
\qquad
h(t):=\psi(t)-\log t+\frac{1}{t-1},
\]
whenever $t>1$ (and $\mathcal R_{n,\alpha}=+\infty$ for $t\le 1$).

Differentiating,
\[
h'(t)=\psi_1(t)-\frac{1}{t}-\frac{1}{(t-1)^2}.
\]
Using \eqref{eq:laplace-1}, \eqref{eq:laplace-2}, \eqref{eq:laplace-3} with $w=1$
we obtain
\begin{equation}
\label{eq:laplace_gprime}
-h'(t)=\int_0^\infty e^{-ts}\,\breve B(s)\,ds;
\qquad
\breve B(s):=1+s e^{s}-\frac{s}{1-e^{-s}}.
\end{equation}
To see that $\breve B(s)>0$ for all $s>0$, set $u=e^s>1$ and set
\[
\breve F(u):= (u-1)\breve B(s)=(u-1)+u(u-2)\log u.
\]
Then one has
\[
\breve F'(u)=(u-1)(2\log u+1)>0,\quad\forall u>1.
\]
Meanwhile $\breve F(1)=0$, and so $\breve F(u)>0$ for all $u>1$.
This implies $\breve B(s)>0$ for all $s>0$ as desired.
\end{proof}

\section{Reverse KL monotonicity in exponential families}
\label{sec:reverse_exp_fam}

In this section we prove Proposition~\ref{prop:reverse_exp_family_monotone} on reverse KL monotonicity in general exponential families, and then present the explicit Gaussian formulas in Proposition~\ref{prop:reverse_known_mean} and \ref{prop:reverse_unknown_mean}.

To generalize the monotonicity of reverse KL divergence, we take advantage of the following domination in the convex order.
This idea goes back to \cite{marshall1965inequality,doob1971martingale}; for other recent applications see
\cite{mattei2023ensembles} and \cite[Theorem 4 and Proposition 5]{manole2023martingale}.

\begin{lemma}
\label{lem:convex_mean_general}
Let $Z_1,Z_2,\dots$ be IID in a finite-dimensional real vector space $V$, and let
$\bar Z_n:=\frac1n\sum_{i=1}^n Z_i$.
Let $\phi:V\to(-\infty,+\infty]$ be convex and assume $\mathbb E[\phi(\bar Z_{n_0})]<\infty$.
Then
\[
\mathbb E[\phi(\bar Z_{n+1})]\le \mathbb E[\phi(\bar Z_n)],\quad\forall n\geq n_0.
\]
If $\phi$ is strictly convex on a convex set supporting $\bar Z_n$ and $Z_1$ is non-degenerate, then the inequality is strict.
\end{lemma}

\begin{proof}
Let 
\[
\bar Z_n^{(-i)}:=\frac1n\sum_{\substack{1\le j\le n+1 \\ j\neq i}} Z_j
\]
be the leave-one-out mean.
Let $I$ be uniform on $\{1,\dots,n+1\}$, independent of all $Z_j$.
Then $\mathbb E[\bar Z_n^{(-I)}\mid Z_1,\dots,Z_{n+1}]=\bar Z_{n+1}$.
Jensen gives
\[
\phi(\bar Z_{n+1})
\le \mathbb E[\phi(\bar Z_n^{(-I)})\mid Z_1,\dots,Z_{n+1}].
\]
Taking outer expectations yields the claim since $\mathbb E[\phi(\bar Z_n^{(-I)})|I=j]=\mathbb E[\phi(\bar Z_n)]$ for each $j$.
\end{proof}

\subsection{Regular exponential families and Bregman form of reverse KL}

Let $\nu$ be a measure on $\mathcal X$ and consider a (minimal, regular) exponential family on $\mathcal X$ given by the densities
\[
p_\theta(x)=\exp\big(\langle \theta, T(x)\rangle - A(\theta)\big)\,h(x),
\qquad \theta\in\Theta\subset\mathbb R^m,
\]
where 
\[
A(\theta)=\log\int \exp(\langle\theta,T(x)\rangle)h(x)\,d\nu(x).
\]
Define the mean map $\mu(\theta):=\mathbb E_\theta[T(X)]=\nabla A(\theta)$.
Let $\mathcal M:=\nabla A(\Theta)$ be the mean-parameter space, and let
\[
A^\ast(\mu):=\sup_{\theta\in\Theta}\{\langle \theta,\mu\rangle - A(\theta)\}
\]
be the Fenchel dual on $\mathcal M$, which is well known to be convex.

\begin{proposition}
\label{prop:reverseKL_bregman}
Let $\theta,\theta_\ast\in\Theta$, and set $\mu=\nabla A(\theta)$, $\mu_\ast=\nabla A(\theta_\ast)$.
Then
\[
D_{\mathrm{KL}}(p_\theta\,\|\,p_{\theta_\ast})
=
A^\ast(\mu)-A^\ast(\mu_\ast)-\langle \nabla A^\ast(\mu_\ast),\,\mu-\mu_\ast\rangle
=: B_{A^\ast}(\mu,\mu_\ast),
\]
the Bregman divergence generated by the convex function $A^\ast$.
In particular, $\mu\mapsto D_{\mathrm{KL}}(p_{\theta(\mu)}\|p_{\theta_\ast})$ is convex on $\mathcal M$.
\end{proposition}

\begin{proof}
We compute directly by expanding the log-density ratio.
Since both $p_\theta$ and $p_{\theta_\ast}$ share the same base
measure $h(x)\,d\nu(x)$, the $h$ terms cancel in $\log(p_\theta/p_{\theta_\ast})$, leaving only the sufficient-statistic
and log-partition contributions:
\[
D_{\mathrm{KL}}(p_\theta\|p_{\theta_\ast})
=\mathbb E_\theta[\langle \theta-\theta_\ast,T(X)\rangle - A(\theta)+A(\theta_\ast)]
=\langle \theta-\theta_\ast,\mu\rangle - A(\theta)+A(\theta_\ast).
\]
By Legendre duality, for $\mu=\nabla A(\theta)$ one has
$A^\ast(\mu)=\langle \theta,\mu\rangle - A(\theta)$, and similarly
$A^\ast(\mu_\ast)=\langle \theta_\ast,\mu_\ast\rangle - A(\theta_\ast)$.
Moreover $\nabla A^\ast(\mu_\ast)=\theta_\ast$. Substituting these identities into the previous display and
rearranging yields the stated Bregman form, which is clearly convex.
\end{proof}

Let $X_1,\dots,X_n\sim p_{\theta_\ast}$ be IID and write $\bar T_n=\frac1n\sum_{i=1}^n T(X_i)$.
Whenever the MLE exists and lies in the interior, it satisfies
\[
\nabla A(\widehat\theta_n)=\bar T_n,
\qquad \iff \qquad \widehat\mu_n:=\nabla A(\widehat\theta_n)=\bar T_n.
\]

 We now prove Proposition~\ref{prop:reverse_exp_family_monotone}.
 
 \begin{proof}[Proof of Proposition~\ref{prop:reverse_exp_family_monotone}]
By Proposition~\ref{prop:reverseKL_bregman},
\[
D_{\mathrm{KL}}(p_{\widehat\theta_n}\|p_{\theta_\ast})
= B_{A^\ast}(\widehat\mu_n,\mu_\ast)
= B_{A^\ast}(\bar T_n,\mu_\ast).
\]
Define $\phi(\mu):=B_{A^\ast}(\mu,\mu_\ast)$.
As a function of $\mu$, this is convex because it is the sum of the convex function $A^\ast(\mu)$ and an affine
function of $\mu$ (the remaining Bregman terms depend on $\mu$ only linearly).

Now $\bar T_n=\frac1n\sum_{i=1}^n T(X_i)$ is the sample mean of the IID vectors $T(X_i)$.
Apply Lemma~\ref{lem:convex_mean_general} with $Z_i:=T(X_i)$ and $\bar Z_n:=\bar T_n$ to conclude
$\mathbb E[\phi(\bar T_{n+1})]\le \mathbb E[\phi(\bar T_n)]$, i.e.\ the expected reverse KL is nonincreasing.
Strictness follows from the strict convexity/non-degeneracy conditions exactly as in Lemma~\ref{lem:convex_mean_general}.
\end{proof}

\subsection{Explicit Reverse KL for Gaussians with Known Mean}

Assume $X_1,\dots,X_n\sim\mathcal N(0,\Sigma_\ast)$ and the MLE is
\[
\widehat\Sigma_n := \frac1n\sum_{i=1}^n X_iX_i^\top.
\]
We derive explicit formulas for the reverse-KL learning curve
\[
\widetilde{\mathcal E}^{(0)}_{n,d}(\Sigma_\ast)
:= \mathbb E\Big[D_{\mathrm{KL}}\big(\mathcal N(0,\widehat\Sigma_n)\,\|\,\mathcal N(0,\Sigma_\ast)\big)\Big]
\]
and again deduce monotonicity.

\begin{proof}[Proof of Proposition~\ref{prop:reverse_known_mean}]
As before we reduce to the identity-covariance case by coordinate change. Set $Y_i:=\Sigma_\ast^{-1/2}X_i$, so
$Y_i\sim\mathcal N(0,I_d)$, and define the empirical covariance in the new coordinates by
\[
\widehat S_n := \Sigma_\ast^{-1/2}\widehat\Sigma_n\Sigma_\ast^{-1/2} = \frac1n\sum_{i=1}^n Y_iY_i^\top.
\]
Since KL divergence is invariant under the invertible change of variables $x\mapsto \Sigma_\ast^{-1/2}x$,
\[
D_{\mathrm{KL}}\big(\mathcal N(0,\widehat\Sigma_n)\,\|\,\mathcal N(0,\Sigma_\ast)\big)
=
D_{\mathrm{KL}}\big(\mathcal N(0,\widehat S_n)\,\|\,\mathcal N(0,I_d)\big).
\]
The Gaussian KL formula gives, for $\Sigma\succ 0$, that
\[
D_{\mathrm{KL}}(\mathcal N(0,\Sigma)\|\mathcal N(0,I_d))=\frac12(\tr(\Sigma)-d-\log\det\Sigma)
\]
and hence
\[
D_{\mathrm{KL}}\big(\mathcal N(0,\widehat S_n)\,\|\,\mathcal N(0,I_d)\big)
=\frac12\Big(\tr(\widehat S_n)-d-\log\det(\widehat S_n)\Big).
\]
Let $W:=\sum_{i=1}^n Y_iY_i^\top$, so $W\sim\mathrm{Wishart}_d(n,I_d)$ and $\widehat S_n=W/n$.
Moreover $\tr(W)=\sum_{i=1}^n \|Y_i\|^2$, so $\mathbb E[\tr(W)]=nd$ and thus $\mathbb E[\tr(\widehat S_n)]=d$.
Therefore the $\tr(\widehat S_n)-d$ term vanishes in expectation, giving
\[
\widetilde{\mathcal E}^{(0)}_{n,d}(\Sigma_\ast)
= -\frac12\,\mathbb E[\log\det(\widehat S_n)]
= -\frac12\Big(\mathbb E[\log\det W]-d\log n\Big).
\]
Applying Proposition~\ref{prop:ElogdetW} yields
\[
\mathbb E[\log\det W]
= d\log 2 + \sum_{j=1}^d \psi\!\Big(\frac{n-j+1}{2}\Big),
\]
so
\begin{equation}
\label{eq:reverse_known_closed}
\widetilde{\mathcal E}^{(0)}_{n,d}(\Sigma_\ast)
=\frac12\,r_d(n),
\qquad
r_d(x):= d\log(x/2) - \sum_{j=1}^d \psi\!\Big(\frac{x-j+1}{2}\Big),\quad x>d-1.
\end{equation}
If $n<d$, then $W$ is singular a.s.,
so $\log\det(\widehat S_n)=-\infty$ and the expectation is $+\infty$.

For monotonicity, differentiate $r_d(x)$ for $x>d-1$:
\[
r_d'(x)=\frac{d}{x}-\frac12\sum_{j=1}^d \psi_1\!\Big(\frac{x-j+1}{2}\Big).
\]
By \eqref{eq:trigamma_lower}, for each $j$,
\[
\frac12\psi_1\!\Big(\frac{x-j+1}{2}\Big)>\frac{1}{x-j+1}.
\]
Summing over $j$ gives
\[
r_d'(x)<\frac{d}{x}-\sum_{j=1}^d\frac{1}{x-j+1}\leq \frac{d}{x}-\frac{d}{x}=0,
\]
since each denominator satisfies $x-j+1\leq x$ and hence $(x-j+1)^{-1}\geq x^{-1}$. Hence $r_d$ is strictly decreasing on $(d-1,\infty)$,
and in particular $r_d(n+1)<r_d(n)$ for all integers $n\ge d$.
\end{proof}

\subsection{Explicit Reverse KL for Gaussians with Unknown Mean}

Assume $X_1,\dots,X_n\sim\mathcal N(\mu_\ast,\Sigma_\ast)$ and use the Gaussian MLEs
\[
\widehat\mu_n:=\bar X,\qquad
\widehat\Sigma_n:=\frac1n\sum_{i=1}^n (X_i-\bar X)(X_i-\bar X)^\top.
\]
Define the reverse-KL learning curve
\[
\widetilde{\mathcal E}^{\mathrm{full}}_{n,d}(\mu_\ast,\Sigma_\ast)
:= \mathbb E\Big[D_{\mathrm{KL}}\big(\mathcal N(\widehat\mu_n,\widehat\Sigma_n)\,\|\,\mathcal N(\mu_\ast,\Sigma_\ast)\big)\Big].
\]

 \begin{proof}[Proof of Proposition~\ref{prop:reverse_unknown_mean}]
Again we use the coordinate change 
\[
Y_i:=\Sigma_\ast^{-1/2}(X_i-\mu_\ast)\sim\mathcal N(0,I_d),\qquad
\bar Y=\frac1n\sum_i Y_i.
\]
Then
\[
\widehat\mu_n-\mu_\ast=\Sigma_\ast^{1/2}\bar Y,
\qquad
\widehat\Sigma_n = \Sigma_\ast^{1/2}\widehat S_n\Sigma_\ast^{1/2},
\qquad
\widehat S_n:=\frac1n\sum_{i=1}^n (Y_i-\bar Y)(Y_i-\bar Y)^\top.
\]
The Gaussian KL formula gives
\begin{align*}
D_{\mathrm{KL}}\big(\mathcal N(\widehat\mu_n,\widehat\Sigma_n)\,\|\,\mathcal N(\mu_\ast,\Sigma_\ast)\big)
&=\frac12\Big(\tr(\Sigma_\ast^{-1}\widehat\Sigma_n)
+(\widehat\mu_n-\mu_\ast)^\top\Sigma_\ast^{-1}(\widehat\mu_n-\mu_\ast)
-d-\log\det(\Sigma_\ast^{-1}\widehat\Sigma_n)\Big)\\
&=\frac12\Big(\tr(\widehat S_n)+\|\bar Y\|^2-d-\log\det(\widehat S_n)\Big),
\end{align*}
which is again independent of $(\mu_*,\Sigma_*)$.

Let 
\[
W:=\sum_{i=1}^n (Y_i-\bar Y)(Y_i-\bar Y)^\top
\]
so $\widehat S_n=W/n$.
By Lemma~\ref{lem:centered_scatter_wishart} (applied to $Y_1,\dots,Y_n\sim\mathcal N(0,I_d)$), we have
\[
W\sim\mathrm{Wishart}_d(n-1,I_d)
\]
and that $W$ is independent of $\bar Y\sim\mathcal N(0,I_d/n)$. Consequently,
\begin{align*}
\mathbb E[\tr(\widehat S_n)] 
&= \frac1n\mathbb E[\tr(W)] = \frac{d(n-1)}{n}=d\Big(1-\frac1n\Big);
\\
\mathbb E[\|\bar Y\|^2]
&=\tr(\mathbb E[\bar Y\bar Y^\top])=\tr(I_d/n)
=\frac dn.
\end{align*}
Thus $\mathbb E[\tr(\widehat S_n)+\|\bar Y\|^2-d]=0$. In particular, after taking expectations the reverse KL reduces
again to a (negative) log-determinant term independent of $(\mu_*,\Sigma_*)$:
\[
\widetilde{\mathcal E}^{\mathrm{full}}_{n,d}(\mu_\ast,\Sigma_\ast)
= -\frac12\,\mathbb E[\log\det(\widehat S_n)]
= -\frac12\Big(\mathbb E[\log\det W]-d\log n\Big).
\]
Applying Proposition~\ref{prop:ElogdetW} with $n$ replaced by $n-1$ yields:
\[
\mathbb E[\log\det W]
= d\log 2 + \sum_{j=1}^d \psi\!\Big(\frac{(n-1)-j+1}{2}\Big)
= d\log 2 + \sum_{j=1}^d \psi\!\Big(\frac{n-j}{2}\Big),
\]
i.e. for $r_d$ in \eqref{eq:reverse_known_closed} we have
$\widetilde{\mathcal E}^{\mathrm{full}}_{n,d}
=r_d(n-1)/2$.
We have already seen that $r_d$ is strictly decreasing on $(d-1,\infty)$; thus $\widetilde{\mathcal E}^{\mathrm{full}}_{n,d}$ is decreasing in $n$ for $n>d$, concluding the proof.
\end{proof}

\small
\medskip
\noindent\textbf{Acknowledgement}
Thanks to Sébastien Bubeck and Mehtaab Sawhney for helpful comments.


\bibliographystyle{alphaabbr}
\bibliography{bib}

\newcommand{\etalchar}[1]{$^{#1}$}
\begin{thebibliography}{BHMM19}

\bibitem[AFS92]{amari1992four}
S.~Amari, N.~Fujita, and S.~Shinomoto.
\newblock Four types of learning curves.
\newblock {\em Neural Computation}, 4(4):605--618, 1992.

\bibitem[Ama93]{amari1993universal}
S.~Amari.
\newblock A universal theorem on learning curves.
\newblock {\em Neural networks}, 6(2):161--166, 1993.

\bibitem[And03]{Anderson2003}
T.~Anderson.
\newblock {An Introduction to Multivariate Statistical Analysis}.
\newblock {\em Wiley series in probability and statistics}, 2003.

\bibitem[BDK{\etalchar{+}}22]{BousquetEtAl2022}
O.~J. Bousquet, A.~Daniely, H.~Kaplan, Y.~Mansour, S.~Moran, and U.~Stemmer.
\newblock Monotone learning.
\newblock In {\em Conference on Learning Theory}, pages 842--866. PMLR, 2022.

\bibitem[BE53]{bateman1953higher}
H.~Bateman and A.~Erd{\'e}lyi.
\newblock {Higher transcendental functions, volume II}.
\newblock {\em Bateman Manuscript Project) McGraw-Hill Book Company}, 410, 1953.

\bibitem[Ber29]{bernstein1929fonctions}
S.~Bernstein.
\newblock Sur les fonctions absolument monotones.
\newblock {\em Acta Mathematica}, 52(1):1--66, 1929.

\bibitem[BHMM19]{belkin2019reconciling}
M.~Belkin, D.~Hsu, S.~Ma, and S.~Mandal.
\newblock Reconciling modern machine-learning practice and the classical bias--variance trade-off.
\newblock {\em Proc. Natl. Acad. Sci. U.S.A.}, 116(32):15849--15854, 2019.

\bibitem[BHX20]{belkin2020two}
M.~Belkin, D.~Hsu, and J.~Xu.
\newblock Two models of double descent for weak features.
\newblock {\em SIAM Journal on Mathematics of Data Science}, 2(4):1167--1180, 2020.

\bibitem[CT99]{cover1999elements}
T.~M. Cover and J.~A. Thomas.
\newblock {\em {Elements of Information Theory}}.
\newblock John Wiley \& Sons, 1999.

\bibitem[DGL97]{devroye1997probabilistic}
L.~Devroye, L.~Gy{\"o}rfi, and G.~Lugosi.
\newblock {\em A Probabilistic Theory of Pattern Recognition}, volume~31.
\newblock Springer Science \& Business Media, 1997.

\bibitem[Doo71]{doob1971martingale}
J.~L. Doob.
\newblock What is a martingale?
\newblock {\em The American Mathematical Monthly}, 78(5):451--463, 1971.

\bibitem[HBM{\etalchar{+}}22]{hoffmann2022training}
J.~Hoffmann, S.~Borgeaud, A.~Mensch, E.~Buchatskaya, T.~Cai, E.~Rutherford, D.~de~Las~Casas, L.~A. Hendricks, J.~Welbl, A.~Clark, et~al.
\newblock Training compute-optimal large language models.
\newblock In {\em 36th International Conference on Neural Information Processing Systems}, 2022.

\bibitem[HKOS91]{haussler1991estimating}
D.~Haussler, M.~Kearns, M.~Opper, and R.~Schapire.
\newblock {Estimating average-case learning curves using Bayesian, statistical physics and VC dimension methods}.
\newblock {\em Advances in Neural Information Processing Systems}, 4, 1991.

\bibitem[HMRT22]{hastie2022surprises}
T.~Hastie, A.~Montanari, S.~Rosset, and R.~J. Tibshirani.
\newblock Surprises in high-dimensional ridgeless least squares interpolation.
\newblock {\em Annals of Statistics}, 50(2):949, 2022.

\bibitem[KBJ19]{JohnsonKotzBalakrishnan1994}
S.~Kotz, N.~Balakrishnan, and N.~L. Johnson.
\newblock {\em Continuous multivariate distributions, Volume 1: Models and applications}, volume~1.
\newblock John Wiley \& Sons, 2019.

\bibitem[KMH{\etalchar{+}}20]{kaplan2020scaling}
J.~Kaplan, S.~McCandlish, T.~Henighan, T.~B. Brown, B.~Chess, R.~Child, S.~Gray, A.~Radford, J.~Wu, and D.~Amodei.
\newblock Scaling laws for neural language models.
\newblock {\em arXiv:2001.08361}, 2020.

\bibitem[Lax14]{lax2014functional}
P.~D. Lax.
\newblock {\em Functional analysis}.
\newblock John Wiley \& Sons, 2014.

\bibitem[LKB23]{loog2023also}
M.~Loog, J.~H. Krijthe, and M.~Bicego.
\newblock Also for $k$-means: more data does not imply better performance.
\newblock {\em Machine Learning}, 112(8):3033--3050, 2023.

\bibitem[LVM19]{loog2019minimizers}
M.~Loog, T.~Viering, and A.~Mey.
\newblock Minimizers of the empirical risk and risk monotonicity.
\newblock {\em Advances in Neural Information Processing Systems}, 32, 2019.

\bibitem[LW21]{li2021minimum}
Y.~Li and Y.~Wei.
\newblock Minimum $\ell_1$-norm interpolators: Precise asymptotics and multiple descent.
\newblock {\em arXiv:2110.09502}, 2021.

\bibitem[Mer14]{merkle2014completely}
M.~Merkle.
\newblock Completely monotone functions: a digest.
\newblock In {\em Analytic Number Theory, Approximation Theory, and Special Functions: In Honor of Hari M. Srivastava}, pages 347--364. 2014.

\bibitem[MG23]{mattei2023ensembles}
P.-A. Mattei and D.~Garreau.
\newblock Are ensembles getting better all the time?
\newblock {\em arXiv preprint arXiv:2311.17885}, 2023.

\bibitem[MM22]{mei2022generalization}
S.~Mei and A.~Montanari.
\newblock The generalization error of random features regression: Precise asymptotics and the double descent curve.
\newblock {\em Comm. Pure Appl. Math.}, 75(4):667--766, 2022.

\bibitem[MP65]{marshall1965inequality}
A.~W. Marshall and F.~Proschan.
\newblock An inequality for convex functions involving majorization.
\newblock {\em Journal of Mathematical Analysis and Applications}, 12(1):87--90, 1965.

\bibitem[MR23]{manole2023martingale}
T.~Manole and A.~Ramdas.
\newblock Martingale methods for sequential estimation of convex functionals and divergences.
\newblock {\em IEEE Transactions on Information Theory}, 69(7):4641--4658, 2023.

\bibitem[Mui09]{Muirhead1982}
R.~J. Muirhead.
\newblock {\em Aspects of multivariate statistical theory}.
\newblock John Wiley \& Sons, 2009.

\bibitem[NKB{\etalchar{+}}21]{nakkiran2021deep}
P.~Nakkiran, G.~Kaplun, Y.~Bansal, T.~Yang, B.~Barak, and I.~Sutskever.
\newblock Deep double descent: Where bigger models and more data hurt.
\newblock {\em Journal of Statistical Mechanics: Theory and Experiment}, 2021(12):124003, 2021.

\bibitem[OH91]{opper1991calculation}
M.~Opper and D.~Haussler.
\newblock Calculation of the learning curve of bayes optimal classification algorithm for learning a perceptron with noise.
\newblock In {\em COLT}, volume~91, pages 75--87, 1991.

\bibitem[PC87]{patarnello1987learning}
S.~Patarnello and P.~Carnevali.
\newblock Learning networks of neurons with boolean logic.
\newblock {\em Europhysics Letters}, 4(4):503, 1987.

\bibitem[Pes22]{pestov2022universally}
V.~Pestov.
\newblock A universally consistent learning rule with a universally monotone error.
\newblock {\em Journal of Machine Learning Research}, 23(157):1--27, 2022.

\bibitem[SST92]{seung1992statistical}
H.~S. Seung, H.~Sompolinsky, and N.~Tishby.
\newblock Statistical mechanics of learning from examples.
\newblock {\em Physical review A}, 45(8):6056, 1992.

\bibitem[TT93]{takahashi1993estimating}
H.~Takahashi and E.~Tomita.
\newblock {Estimating learning curves by PAC-learnability criterion}.
\newblock In {\em Proceedings of 1993 International Conference on Neural Networks (IJCNN-93-Nagoya, Japan)}, volume~2, pages 1641--1644. IEEE, 1993.

\bibitem[VL22]{viering2022shape}
T.~Viering and M.~Loog.
\newblock The shape of learning curves: a review.
\newblock {\em IEEE Transactions on Pattern Analysis and Machine Intelligence}, 45(6):7799--7819, 2022.

\bibitem[VML19]{VieringLoog2019}
T.~Viering, A.~Mey, and M.~Loog.
\newblock Open problem: Monotonicity of learning.
\newblock In {\em Conference on Learning Theory}, pages 3198--3201. PMLR, 2019.

\bibitem[YHB{\etalchar{+}}21]{yang2021tuning}
G.~Yang, E.~Hu, I.~Babuschkin, S.~Sidor, X.~Liu, D.~Farhi, N.~Ryder, J.~Pachocki, W.~Chen, and J.~Gao.
\newblock Tuning large neural networks via zero-shot hyperparameter transfer.
\newblock {\em Advances in Neural Information Processing Systems}, 34:17084--17097, 2021.

\end{thebibliography}

\end{document}